\documentclass[11pt]{article}
\usepackage[utf8]{inputenc}
\usepackage{amsfonts}
\usepackage{amssymb}
\usepackage{fullpage}
\usepackage{lipsum}
\usepackage{comment}
\usepackage[]{algorithm2e}
\usepackage{amssymb}
\usepackage{tipa}
\usepackage[skip=20pt]{parskip}
\usepackage[shortlabels]{enumitem}
\usepackage{fancyhdr}
\usepackage{amsthm}
\usepackage{makecell}
\usepackage{amsmath}
\usepackage{mathrsfs} 
\usepackage{url}
\usepackage{biblatex}
\usepackage{bbold}
\renewbibmacro{in:}{}
\usepackage{verbatim}
\usepackage{setspace}
\usepackage{mathtools}
\usepackage{tikz}
\usepackage{multicol}
\usepackage{xfrac}
\usepackage{bm}
\usepackage{yfonts}
\usepackage{underscore}
\usepackage[super]{nth}
\usepackage{textcomp}
\usepackage{pdfpages}
\usepackage[UKenglish]{babel}
\usepackage[UKenglish]{isodate}
\usepackage{dsfont}
\usepackage{graphicx}
\usepackage[rightcaption]{sidecap}
\usepackage{wrapfig}
\usepackage[bookmarks=true]{hyperref}

\newcommand{\mathleft}{\@fleqntrue\@mathmargin0pt}
\newcommand{\mathcenter}{\@fleqnfalse}

\theoremstyle{plain}
\newtheorem{theorem}{Theorem}[section]

\theoremstyle{definition}
\newtheorem{definition}[theorem]{Definition}
\newtheorem{corollary}[theorem]{Corollary}
\newtheorem{lemma}[theorem]{Lemma}
\newtheorem{proposition}[theorem]{Proposition}

\newtheorem{remark}[theorem]{Remark}

\newtheorem{question}[theorem]{Question}
\newtheorem{conjecture}[theorem]{Conjecture}
\newtheorem*{acknowledgement}{Acknowledgement}

\makeatletter  
\def\colvec#1{\expandafter\colvec@i#1,,,,,,,,,\@nil}
\def\colvec@i#1,#2,#3,#4,#5,#6,#7,#8,#9\@nil{%
  \ifx$#2$ \begin{pmatrix}#1\end{pmatrix} \else
    \ifx$#3$ \begin{pmatrix}#1\\#2\end{pmatrix} \else
      \ifx$#4$ \begin{pmatrix}#1\\#2\\#3\end{pmatrix}\else
        \ifx$#5$ \begin{pmatrix}#1\\#2\\#3\\#4\end{pmatrix}\else
          \ifx$#6$ \begin{pmatrix}#1\\#2\\#3\\#4\\#5\end{pmatrix}\else
            \ifx$#7$ \begin{pmatrix}#1\\#2\\#3\\#4\\#5\\#6\end{pmatrix}\else
              \ifx$#8$ \begin{pmatrix}#1\\#2\\#3\\#4\\#5\\#6\\#7\end{pmatrix}\else
                 \PackageError{Column Vector}{The vector you tried to write is too big, use pmatrix instead}{Try using the pmatrix environment}
              \fi
            \fi
          \fi
        \fi
      \fi
    \fi
  \fi 
}  
\makeatother

\emergencystretch=\maxdimen
\graphicspath{ {./pictures/} }

\title{Arithmetic Properties of Character Degrees and \\ the Generalised Knutson Index} 
\author{Diego Martín Duro}
\date{\vspace{-1ex}}

\begin{document}

\maketitle

\begin{abstract}
In this paper, we introduce the generalised Knutson Index and compute it for the special linear groups and projective special linear groups of degree two by computing the lowest common multiple of the degrees of their irreducible representations. We also classify all alternating and symmetric groups such that the lowest common multiple of the degrees of their irreducible representations equals the order groups, which yields a lower bound on the generalised Knutson Indices of these groups.
\end{abstract}

\section{Introduction}
The set of virtual characters over a group, i.e., integer linear combinations of irreducible characters, forms a ring with direct sum as addition and tensor product as multiplication. Donald Knutson conjectured the following in 1973 \cite{knutson}.

\begin{conjecture} [Knutson]
For every irreducible character $\chi$ of a finite group $G$, there exists a virtual character $\lambda \in \mathbb{Z}[\text{Irr}(G)]$ such that $\chi \otimes \lambda = \rho_\text{reg}$, where $\rho_\text{reg}$ is the regular character.
\end{conjecture}

Savitskii observed in 1992 that this conjecture failed for $SL_2(\mathbb{F}_5)$ \cite{savitskii}. We found further counter-examples to this conjecture and introduced the Knutson Index of a group as a measure of Knutson's Conjecture failure along with a few algebraic properties and classification results \cite{diego}.

\begin{definition}
The Knutson Index of a finite group $G$ is defined as the smallest positive integer $n$ such that for every irreducible character $\chi$ of $G$ there exists a virtual $\lambda$ such that $\chi \otimes \lambda = n \rho_\text{reg}$. We denote it by $\mathcal{K}(G)$.
\end{definition}

We now introduce the generalised Knutson Index.

\begin{definition}
A character $\chi$ is $\rho$-invertible if there exists a virtual character $\lambda$ such that $\chi \otimes \lambda = \rho$. Let $\rho$ be a character of the smallest degree such that every irreducible character is $\rho$-invertible. The generalised Knutson Index of $G$ is
$$ \mathcal{K}'(G) = \frac{\rho(\text{id})}{|G|} \; .$$

\end{definition}

\begin{remark}
This notion can be generalised to an arbitrary semisimple tensor category where the dimension used is the Frobenius-Perron dimension \cite{etingof}.

Let $\mathcal{C}$ be a semisimple tensor category with simple objects $M_1, \ldots M_n$. Let $M$ be the object of the smallest Frobenius-Perron dimension such that all simple objects are $M$-invertible. The generalised Knutson of the category is
$$ \mathcal{K}'(\mathcal{C}) = \frac{\mathcal{FP}(\mathcal{C})}{\mathcal{FP}(M)} \; . $$
If $\mathcal{C}$ is the category of representations of a finite group $G$ then $\mathcal{K}'(\mathcal{C}) = \mathcal{K}'(G)$.
\end{remark}

In this paper, we start by establishing preliminary results of the generalised Knutson Index and first compute it for $SL_2(q)$ and $PSL_2(q)$ and then for $S_n$ and $A_n$. For this purpose, we will also study $L(G)$, the lowest common multiple of the degrees of all irreducible characters of a group $G$ and the existences of zeros in every non-trivial column of a character table.

\begin{acknowledgement}
I would like to express my sincere appreciation towards my PhD supervisor Professor Dmitriy Rumynin for his continued guidance and corrections that have contributed to this paper. This work was supported by the UK Engineering and Physical Sciences Research Council (EPSRC) grant EP/T51794X/1.
\end{acknowledgement}

\section{Preliminary results}
\begin{proposition} \label{lcm}
Let $L(G)$ be the lowest common multiples of the degrees of all irreducible characters of $G$. Then $\mathcal{K}'(G) \geq L(G) / |G|$. In particular, if $L(G) = |G|$ we have that $\mathcal{K}'(G) \geq 1$.
\end{proposition}

\begin{proof}
Let $\rho$ be the character of the smallest degree such that all irreducible characters are $\rho$-invertible. As $\chi(\text{id})$ divides $\rho(\text{id})$ for any irreducible character $\chi$ of $G$ we have that $L(G)$ divides $\rho(\text{id})$. We conclude that $\mathcal{K}'(G) \geq L(G) / |G|$.
\end{proof}

Recall that if $\mathcal{K}(G) = 1$, then we say that $G$ is a group of Knutson type. Note also that $\mathcal{K}'(G) \leq \mathcal{K}(G)$ since every irreducible character is $\mathcal{K}(G) \rho_\text{reg}$-invertible and, in particular, for groups of Knutson type, we have that $\mathcal{K}'(G) \leq 1$.

\begin{corollary}
If $G$ is of Knutson type and $L(G) = |G|$ then $\mathcal{K}'(G) = 1$.
\end{corollary}

We are therefore interested in classifying all groups such that the lowest common multiple of all its irreducible characters equals the order of the group. In this paper, we will answer the question for the alternating groups $A_n$ and symmetric group $S_n$.

\begin{proposition} \label{zeros}
If every non-trivial column of the character table of a group $G$ has a zero, we have that $\mathcal{K}'(G) = \mathcal{K}(G)$.
\end{proposition}

\begin{proof}
Suppose that all irreducible characters of $G$ are $\rho$-invertible and let $g \in G^{\setminus \text{\{id\}}}$. We have that $\chi(g) = 0$ for some $\rho$-invertible character so $\rho(g) = 0$ whenever $g \neq \text{id}$. Hence $\rho$ is a multiple of the regular character and we conclude that $\mathcal{K}'(G) = \mathcal{K}(G)$.
\end{proof}

\section{Special linear group}
Recall that the Knutson Indices of $SL_2(q)$ and $PSL_2(q)$ are known \cite{diego}:
\begin{equation*}
\mathcal{K}(SL_2(q)) = 
    \begin{cases}
      1 & \text{if $q = 2^n$ or $3$} \\
      2 & \text{otherwise}
    \end{cases} \qquad \qquad \mathcal{K}(PSL_2(q)) = 
    \begin{cases}
      1 & \text{if $q = 2^n$ or $q = 2^n \pm 1$} \\
      2 & \text{otherwise}
    \end{cases}
\end{equation*}

So these values are also upper bounds of the generalised Knutson Indices. 

\begin{theorem}
For $PSL_2(q)$, the generalised Knutson Index coincides with the Knutson Index.
\begin{equation*}
\mathcal{K}'(PSL_2(q)) = \mathcal{K}(PSL_2(q)) = 
    \begin{cases}
      1 & \text{if $q = 2^n$ or $q = 2^n \pm 1$} \\
      2 & \text{otherwise}
    \end{cases}
\end{equation*}
\end{theorem}

\begin{proof}
Every non-trivial column of the character table of $PSL_2(q)$ has a zero entry \cite{dornhoff} so by Proposition \ref{zeros} we have that $\mathcal{K}'(PSL_2(q)) = \mathcal{K}(PSL_2(q))$.
\end{proof}

\begin{proposition}
For $q \geq 4$, we have that
\begin{equation*}
\mathcal{K}'(SL_2(q)) \geq 
    \begin{cases}
      \frac{1}{2} & \text{if $q$ odd,} \\
      1 & \text{if $q$ even.}
    \end{cases}
\end{equation*}
\end{proposition}

\begin{proof}
We can see from the character tables of these groups \cite{dornhoff} that for $q \geq 4$
\begin{equation*}
  L(SL_2(q)) =
    \begin{cases}
      \frac{(q+1) q (q-1)}{2} & \text{if $q$ odd,} \\
      (q+1) q (q-1) & \text{if $q$ even.}
    \end{cases} 
\end{equation*}

So the result follows by Proposition \ref{lcm}.    
\end{proof}

\begin{theorem}
For $q \geq 4$, we have that $\mathcal{K'}(SL_2(q)) = 1$.
\end{theorem}

\begin{proof}
The result follows directly for $q$ even by the previous proposition. Now suppose that $q \geq 5$ is odd. We know that some characters are not regular invertible. However, we can show that all irreducible characters are $\rho$-invertible where 
$$ \rho = {\bigoplus_{\substack{\chi \in \text{Irr}(G) \\
\chi(-\text{id}) = \chi(\text{id})}}} 2 \ \chi(\text{id}) \ \chi $$ 
    
The values of this character are
$$ \rho(g) = 
\begin{cases}
    |G| & \text{if } g = \pm \text{id} \\
    0 & \text{otherwise}
\end{cases} $$

 We can construct an explicit $\rho$-inverse for each irreducible character depending on $q$ modulo $4$. Let us first introduce the following notation for the irreducible characters of $SL_2(q)$ \cite{dornhoff}. Let $\mathbb{1}_G$ be the character of degree $1$, $\eta_1, \eta_2$ of degree $\frac{q-1}{2}$, $\xi_1, \xi_2$ of degree $\frac{q+1}{2}$, $\theta_j$ of degree $q-1$, $\Psi$ of degree $q$ and $\chi_i$ of degree $q+1$, where $1 \leq j \leq (q-1)/2$, $1 \leq i \leq (q-3)/2$ and the even indices correspond to the characters that restrict to $PSL_2(q)$.

\begin{table}[htbp]
\centering
\renewcommand{\arraystretch}{1.5}
\begin{tabular}{|c|c|c|}
\hline
$\rho$-inverses for & If $q \equiv 1$ mod $4$ & If $q \equiv 3$ mod $4$ \\
\hline
$\eta_1, \eta_2$ & $ 2 \eta_1 + 2 \eta_2 + 4 \sum_{j \ \text{odd}} \theta_j + (q+1) \chi_1 $ & $ (q-1) \mathbb{1}_G + 2 \eta_1 + 2 \eta_2 + 4 \sum_{j \ \text{even}} \theta_j + (q+3) \Psi $ \\
$\xi_1, \xi_2$ &  $4 \mathbb{1}_G + 2 \xi_1 + 2 \xi_2 + (q+1) \theta_2 + 4 \sum_{i \ \text{even}} \chi_i$ & $ 2 \xi_1 + 2 \xi_2 + (q-1) \theta_1 + 4 \sum_{i \ \text{odd}} \chi_i $ \\
$\theta_j$ with $j$ odd & $ \eta_1 + \eta_2 + 2 \sum_{j \ \text{odd}} \theta_j + \frac{q+1}{2} \chi_1 $ & $ \frac{q+1}{2} \xi_1 + \frac{q+1}{2} \xi_2 + 2 \sum_{j \ \text{odd}} \theta_j $ \\
$\theta_j$ with $j$ even & $-2 \mathbb{1}_G + \frac{q+3}{2} \xi_1 + \frac{q+3}{2} \xi_2 + 2 \sum_{j \ \text{even}} \theta_j$  & $ \frac{q-1}{2} \mathbb{1}_G + \eta_1 + \eta_2 + 2 \sum_{j \ \text{even}} \theta_j + \frac{q+3}{2} \Psi $ \\
$\Psi$ & $ -2 \mathbb{1}_G + 4 \sum_{j \ \text{even}} \theta_j + 2 \Psi $ & $ -2 \mathbb{1}_G + 2 \eta_1 + 2 \eta_2 + 4 \sum_{j \ \text{even}} \theta_j + 2 \Psi $ \\
$\chi_i$ with $i$ odd & $ \frac{q-1}{2} \eta_1 + \frac{q-1}{2} \eta_2 + 2 \sum_{i \ \text{odd}} \chi_i $ & $ \xi_1 + \xi_2 + \frac{q-1}{3} \chi_1 + 2 \sum_{i \ \text{odd}} \chi_i $ \\
$\chi_i$ with $i$ even & $ 2 \mathbb{1}_G + \xi_1 + \xi_2 + \frac{q+1}{2} \theta_2 + 2 \sum_{i \ \text{even}} \chi_i $ & $ 2 \mathbb{1}_G + \frac{q+1}{2} \eta_1 + \frac{q+1}{2} \eta_2 + 2 \sum_{i \ \text{even}} \chi_i $ \\
\hline
\end{tabular}
\end{table}

So we conclude that $\mathcal{K'}(SL_2(q)) \leq 1$.

We now show that $\mathcal{K'}(SL_2(q)) \geq 1$. Note that every column in the character table of $SL_2(q)$ has a zero, except the columns corresponding to $\pm \text{id}$. So if all characters are $\rho$-invertible, we must have that $\rho(g) = 0$ whenever $g \neq \pm \text{id}$. The only characters satisfying this and of degree smaller than $|SL_2(q)|$ are
$$ \rho^{\pm} (g) = \begin{cases}
    \ \ |G|/2 & \text{if } g = \text{id}, \\
    \pm |G|/2 & \text{if } g = - \text{id}, \\
    \quad 0 & \text{otherwise.}
\end{cases} $$

\newpage
Note that $\rho^+ \oplus \rho^- = \rho_\text{reg}$, where $\rho^+$ is the sum of all irreducible constituents of $\rho_{reg}$ that are positive when evaluated at $-\text{id}$ and $\rho^-$ is the sum of those which are negatives when evaluated at $-\text{id}$.

Suppose that all irreducible characters are $\rho^\pm$-invertible. For $q \equiv 1$ mod $4$, it follows from the character table that we can choose $\theta_1$ and $\theta_2$ to be two irreducible characters of degree $q-1$ such that $\theta_1(-\text{id}) = - \theta_2(-\text{id})$. By our assumptions, they are both $\rho^\pm$-invertible, so there exist $\lambda_1$ and $\lambda_2$ such that
$$ \theta_1 \otimes \lambda_1 = \rho^\pm \quad \text{and} \quad \theta_2 \otimes \lambda_2 = \rho^\pm. $$
But by our choice of $\theta_1$ and $\theta_2$ we must also have that 
$$ \theta_1 \otimes \lambda_2 = \rho^\mp \quad \text{and} \quad \theta_2 \otimes \lambda_1 = \rho^\mp. $$
So
$$ \theta_{1,2} \otimes (\lambda_1 \oplus \lambda_2) = \rho^\pm \oplus \rho^\mp = \rho_\text{reg} \; . $$
This contradicts our previous result that no characters of degree $q-1$ are regular invertible, so not all irreducible characters of $G$ can be $\rho^\pm$-invertible. The same proof works for $q \equiv 3$ mod $4$ by taking a character of degree $q+1$. We conclude that $\mathcal{K'}(SL_2(q)) = 1$.
\end{proof}

\begin{remark}
We can compute the generalised Knutson Indices for $SL_2(2)$ and $SL_2(3)$ directly by finding a lower bound with $L(G)$ and then showing that it is an equality by constructing $\rho$ explicitly. We obtain that $\mathcal{K}'(SL_2(2)) = 1/3$ and $\mathcal{K}'(SL_2(3))=1/2$.
\end{remark}

\section{Alternating and Symmetric Groups}
We will now discuss the generalised Knutson Indices of the alternating and symmetric groups. We have computed that $S_n$ is of Knutson type for $n \leq 16$, so we propose the following conjecture.

\begin{conjecture}
$S_n$ is of Knutson type for every $n$.
\end{conjecture}

This conjecture would imply that $\mathcal{K}'(S_n) \leq 1$. It is possible to show that $\mathcal{K}(S_n) \leq \mathcal{K}(A_n) \leq 2 \cdot \mathcal{K}(S_n)$. We will now study sufficient conditions on $n$ for $\mathcal{K}'(S_n) \geq 1$.

Recall that irreducible characters of $S_n$ are in bijective correspondence with the Young tableaux of size $n$ \cite[Chapter 2]{james}. The hook length $H_{i,j}$ of the box $(i, j)$ is the size of the set containing the boxes at $\{(k, j) | k \geq i\} \cup \{(i, k) | k \geq j\}$. In other words, the number of boxes directly below or to the right plus the box itself. It is known that
the degree of the irreducible character corresponding to a Young tableau is 
$$ \frac{n!}{\prod_{i,j} H_{i,j}} \;.$$ 

\begin{definition}
A $t$-core partition of $n$ is a partition such that none of the hook lengths is a multiple of $t$.
\end{definition}

\begin{proposition}
$L(S_n) = n!$ if and only if for every prime $p$ there exists a $p$-core partition of $n$.
\end{proposition}

We are therefore interested in finding positive integers $n$ such that $n$ admits a p-core partition for every prime $p$.

\begin{lemma}
There exists a $2$-core partition of $n$ if and only if $n$ is a triangular number.
\end{lemma}

\begin{proof}
Suppose that $\lambda = (\lambda_1, \ldots, \lambda_r)$ is a $2$-core partition of $n$. We must have that $\lambda_i - \lambda_{i+1} \leq 1$ for all $i$'s, as otherwise, the second last box of $\lambda_i$ has hook length $2$. Note also that if $\lambda_i = \lambda_{i+1}$, then the second last row of this length ends with a box of hook length two. So we conclude that the partition is of the form $(r, r-1, \ldots, 1)$ and $n$ must be a triangular number.

Conversely, if $n$ is triangular, then $n = \sum_{i=1}^r i$ for some positive integer $r$. We obtain the partition $\lambda = (r, r-1,\ldots, 1)$ of $n$ and since all hook lengths of this partition are odd, $\lambda$ is a $2$-core.
\end{proof}

\begin{definition}
Löschian numbers are integers that can be written in the form $X^2 + XY + Y^2$ for some integers $X$ and $Y$.
\end{definition}

\begin{lemma}
There exists a $3$-core partition of $n$ if and only if $3n + 1$ is a Löschian number.
\end{lemma}

\begin{proof}
Let $c_3(n)$ be the number of $3$-core partitions of $n$. It is known \cite{granville} that $c_3(n) = \sigma_3(3n+1)$ where \begin{equation*}
  \sigma_3(n) =
    \begin{cases}
      0 & \text{if $n \equiv 0$ mod $3$}, \\
      \sum_{d|n} \left(\frac{d}{3}\right) & \text{if $n \equiv 1, 2$ mod $3$}.
    \end{cases}      
\end{equation*}

So $c_3(n) > 0$ if and only if in the prime factorisation of $3n + 1$, every prime $p \equiv 2$ mod $3$ appears with even multiplicity. These are precisely the Löschian numbers \cite{marshall}.
\end{proof}

\begin{lemma}
For a prime $p \geq 5$, there always exists a $p$-core partition of $n$.
\end{lemma}

The cases $p=5$ and $p=7$ have been proven in \cite{garvan}, for $p = 11$ in \cite{ono94} and for $p \geq 13$ in \cite{granville}. \\

\begin{theorem}
There exist infinitely many integers $n$ such that $L(S_n) = n!$ and this sequence is precisely the intersection of the sequence of triangular numbers and the sequence of numbers of the form $X^2 + X + XY + Y + Y^2$. 
\end{theorem}

\begin{proof}
We begin by showing that $n$ is of the form $X^2 + X + XY + Y + Y^2$ if and only if $3n+1$ is Löschian.

Suppose $n = X^2 + X + XY + Y + Y^2$. Then for $A = X - Y$ and $B = X + 2Y + 1$ we have that 
$$ A^2 + A B + B^2 = (X - Y)^2 + (X - Y)(X + 2Y + 1) + (X + 2Y + 1)^2 = $$
$$ 3X^2 + 3X + 3XY + 3Y + 3Y^2 + 1 = 3n + 1 $$
so $3n + 1$ is Löschian.

Conversely, suppose that $3n + 1$ is Löschian. So $3n + 1 = A^2 + AB + B^2$ for some integers $A$ and $B$. Note that $A \not\equiv B$ so we can assume that $A - B \equiv 1$ mod $3$. Let $ X = \frac{A + 2B - 1}{3}$ and $ Y = \frac{-A + B -1}{3} $. Then 
$$ X^2 + X + XY + Y + Y^2 = \left( \frac{A + 2B - 1}{3} \right)^2 + \left( \frac{A + 2B - 1}{3} \right) \left( \frac{A - B -1}{3} \right) + \left( \frac{A - B -1}{3} \right)^2 = $$
$$ \frac{A^2 + AB + B^2}{3} = n $$
and we conclude that $n$ can be written in the form  $X^2 + X + XY + Y + Y^2$.

Now note that for all squares there exists a $3$-core partition and it is well-known that there are an infinite number of triangular squares \cite{pietenpol}. We therefore conclude that the sequence is infinite.
\end{proof}

\begin{corollary}
The integer sequence such that $L(S_n) = n!$ is given by
$$ 1, 6, 10, 21, 36, 66, 105, 120, 136, 190, \ldots \qquad \text{\cite{A363675}}$$
For all these numbers, we have that $\mathcal{K}'(S_n) \geq 1$. 
\end{corollary}

We will now answer the same question for $A_n$. We can show computationally that $A_n$ is of Knutson type for all $n \leq 16$ except for $12, 13$ and $15$ where characters with Knutson Index two appear. The representations of $A_n$ can be obtained from the ones of $S_n$.

\begin{proposition} \cite[Theorem 2.5.7]{james}
If $\chi$ is an irreducible character of $S_n$ with a partition that is self-conjugate, then its restriction to $A_n$ splits into two irreducible characters of $A_n$ of the same degree. If $\chi$ is not self-conjugate, then $\chi$ and its conjugate restrict to the same irreducible character of $A_n$.
\end{proposition}

Let $L(A_n)$ be the lowest common multiple of the degrees of all irreducible characters of $A_n$. It follows from the previous proposition that $L(A_n) = L(S_n)$ or $L(A_n) = L(S_n)/2$. So if $L(S_n) = n!$ then $L(A_n) = n!/2 = |A_n|$. However, there are more cases with $L(A_n) = n!/2 = |A_n|$. We are interested in partitions with a unique hook length of two.

\begin{proposition}
There exists a partition of $n$ with a unique hook of even length if and only if $n = m(m+1)/2 + 2$.
\end{proposition}

\begin{proof}
Let $\lambda = (\lambda_1, \ldots, \lambda_r)$ be a partition with a unique even hook length. Given $\lambda_1$ we see that $\lambda_2 = \lambda_1 - 1$ or $\lambda_2 = \lambda_1 - 3$ as otherwise we create multiple even hook lengths. 

In the first case, as long as $\lambda_i \geq 2$ we are now forced to take $\lambda_{i+1} = \lambda_i - 1$ as otherwise the hook lenght of the second last box of $\lambda_i$ is two and of the third last box of $\lambda_{i+1}$ is four. When we reach $\lambda_i = 1$ we see that we need exactly three rows of this length to get exactly one hook with even length in the first column. We obtain that the partition must be $(\lambda_1, \lambda_1-1, \lambda_1-2, \ldots, 3, 2, 1, 1, 1)$.

In the second case, we get that the second last hook of the first row is of length two, so we need $\lambda_{i+1} = \lambda_{i} -1$ for $i \geq 2$ as otherwise we create a second hook of even lenght at the second last box of $\lambda_1$. We obtain the partition $(\lambda_1, \lambda_1 - 3, \lambda_1 - 4, \ldots, 3, 2, 1)$ which is the transpose of the previous partition. 

We conclude in both cases that $n$ is a triangular number plus two.
\end{proof}

\begin{theorem}
We have that $L(A_n) = |A_n|$ if and only if there is a $3$-core partition of $n$, so if and only if $n$ is of the form $X^2 + X + X Y + Y + Y^2$ and either $n$ is triangular or a triangular number plus two.
\end{theorem}

\begin{corollary} \label{An seq}
The integer sequence such that $L(A_n) = n!/2$ contains the one for $L(S_n) = n!$ and is given by 
$$ 1, 2, 5, 6, 8, 10, 12, 17, 21, 30, 36, 57, \ldots \qquad \text{\cite{A363676}}$$ 
For all these numbers, we have that $\mathcal{K}'(S_n) \geq 1$.
\end{corollary}

So we have found sufficient requirements on $n$ such that $\mathcal{K}'(S_n) \geq 1$ and $\mathcal{K}'(A_n) \geq 1$. We also know that if $S_n$ and $A_n$ contain a zero in every non-trivial column then $\mathcal{K}'(S_n) = \mathcal{K}(S_n) \geq 1$ and $\mathcal{K}'(A_n) = \mathcal{K}(A_n) \geq 1$.

\begin{proposition}
For $n \geq 3$, the character table of $A_n$ has a zero in every non-trivial column if and only if $S_n$ does.
\end{proposition}

\begin{proof}
For one direction, note that all odd permutations vanish for self-transpose characters of $S_n$ and the other direction follows from the formula for splitting characters  \cite[Theorem 2.5.12]{james}.
\end{proof}

\begin{proposition}
A non-vanishing conjugacy class of $S_n$ has cycle shape $(3^a, 2^b)$, where $b$ is even. Furthermore, if $3n+1$ is Löschian then $a=0$ and if $n$ or $n-2$ is triangular then $b=0$.
\end{proposition}

\begin{proof}
It follows from the Murnaghan-Nakayama Formula \cite[Theorem 2.4.7]{james} that if for some $s \geq 0$ and $t \geq 2$, there exists a $t$-core partition of $n-st$ then all conjugacy classes containing more than $s$ cycles of length $t$ vanish for some character \cite{Morotti}. As for all $t \geq 4$, there exists a $t$-core partition of $n$, every element containing a $t$-cycle vanishes. If $3n+1$ is Löschian, then there exists a $3$-core partition of $n$, so every element containing a $3$-cycle vanishes. If $n$ is a triangular number, then every element containing a $2$-cycle vanishes and if $n-2$ is a triangular number, then every element containing two $2$-cycles vanishes, so $b \leq 1$, but as $(3^a, 2^b)$ is even this implies that $b = 0$.
\end{proof}

\begin{corollary} \label{cor zeros}
For all integers in the sequence of Corollary \ref{An seq}, we have that all non-trivial columns of $A_n$ and $S_n$ have a zero and therefore $\mathcal{K}'(A_n), \mathcal{K}'(S_n) \geq 1$.
\end{corollary}

However, if we compute the sequence of integers $n$ such that $S_n$ has a zero for every non-trivial conjugacy class, we obtain more values: $$ 1, 5, 6, 8, 9, 10, 12, 14, 17, 21, 28, 30, 32, 34, 36, 37, 38 \ldots \qquad \text{\cite{A363701}} $$
So we have found the additional values to the ones already known following from Corollary \ref{cor zeros}. These are $9, 14, 28, 32, 34, 37, 38$. We propose the following question.

\begin{question}
Does the sequence of integers $n$ such that $S_n$ has a zero for every non-trivial conjugacy class follow any pattern and are these additional values finite or infinite?
\end{question}

The following table contains the generalised Knutson Indices for $A_n$ and $S_n$ for $n \leq 10$.

\begin{table}[htbp]
\centering
\renewcommand{\arraystretch}{1.5}
\begin{tabular}{|c|c|c|c|c|}
\hline
$n$ & $L(S_n)/n!$ & \makecell{Zeros in every \\ non-trivial column} & $\mathcal{K}'(S_n)$ & $\mathcal{K}'(A_n)$ \\
\hline
$3$ & $1/3$ & $\times$ & $1/3$ & $1/3$ \\
$4$ & $1/4$ & $\times$ & $1/4$ & $1/4$ \\
$5$ & $1/2$ & $\checkmark$ & $1$ & $1$ \\
$6$ & $1$ & $\checkmark$ & $1$ & $1$ \\
$7$ & $1/12$ & $\times$ & $1/12$ & $1/12$ \\
$8$ & $1/2$ & $\checkmark$ & $1$ & $1$ \\
$9$ & $1/8$ & $\checkmark$ & $1$ & $1$ \\
$10$ & $1$ & $\checkmark$ & $1$ & $1$ \\
\hline
\end{tabular}
\end{table}

So we propose the following conjecture.

\begin{conjecture}
If $S_n$ has a zero in every non-trivial column then $\mathcal{K}'(S_n) = \mathcal{K}'(A_n) = 1$ and otherwise $\mathcal{K}'(S_n) = \mathcal{K}'(A_n) = L(S_n)/n!$. 
\end{conjecture}

\newpage
\bibstyle{numeric}
\printbibliography
\end{document}